\documentclass[12pt, twoside, leqno]{article}

\usepackage{amsmath,amsthm}
\usepackage{amssymb}

\usepackage{enumerate}

\usepackage{graphicx}

\newtheorem{thm}{Theorem}[section]

\newtheorem{lem}[thm]{Lemma}

\newtheorem{mainthm}[thm]{Main Theorem}

\theoremstyle{definition}

\numberwithin{equation}{section}

\begin{document}

%%%%% To ease editing, for IMPAN journals add:

\baselineskip=17pt

%%%%%%%%%%%%%%%%

\title{On the Ces\`aro average of the ``Linnik numbers"}

\author{Marco Cantarini\\
Universit\`a di Ferrara\\ 
Dipartimento di Matematica e Informatica\\
Via Machiavelli, 30 \\
44121 Ferrara, Italy\\
E-mail: marco.cantarini@unife.it}

\date{}

\maketitle

%% Classification and key words; note that the 2010 classification is used:

\renewcommand{\thefootnote}{}



\renewcommand{\thefootnote}{\arabic{footnote}}
\setcounter{footnote}{0}

%%%%%%%%

\begin{abstract}
Let $\Lambda$ be the von Mangoldt function and $$r_{Q}\left(n\right)=\sum_{m_{1}+m_{2}^{2}+m_{3}^{2}=n}\Lambda\left(m_{1}\right)$$ be the counting function for the numbers that can be written as sum of a prime and two squares (that we will call “Linnik numbers”, for brevity). Let $N$ a sufficiently large integer. We prove that for $k>3/2$ we have $$\sum_{n\leq N}r_{Q}\left(n\right)\frac{\left(N-n\right)^{k}}{\Gamma\left(k+1\right)}=M\left(N,k\right)+O\left(N^{k+1}\right)$$ where $M\left(N,k\right)$ is essentially a weighted sum, over non-trivial zeros of the Riemann zeta function, of Bessel functions of complex order and real argument. We also prove that with this technique the bound $k>3/2$ is optimal.
\end{abstract}

\section{Introduction}
We continue the recent work of Languasco and Zaccagnini on additive problems with prime summands. In \cite{langzac1} and \cite{langzac2} they study the Ces\`aro weighted explicit formula for the Goldbach numbers (the integers that can be written as sum of two primes) and for the Hardy-Littlewood numbers (the integers that can be written as sum of a prime and a square). In a similar manner, we will study a Ces\`aro weighted explicit formula for the integers that can be written as sum of a prime and two squares. We will obtain an asymptotic formula with a main term and more terms depending explicitly on the zeros of the Riemann zeta function.
The study of these numbers is classical. For example Hardy and Littlewood in \cite{HarLit} studied the number of solutions of the equation \[n=p+a^{2}+b^{2}\]
 and Linnik in \cite{Linnik} derived an asymptotic formula for the
number of representations of these numbers. Similar averages of arithmetical
functions are common in literature, see, e.g., Chandrasekharan - Narasimhan
\cite{ChaNa} and Berndt \cite{Berndt} who built on earlier classical
work. For our work we will need the Bessel functions $J_{v}\left(u\right)$
of complex order $v$ and real argument $u$. For their definition
and main properties we refer to Watson \cite{Wat}, but we recall
that they were introducted by Daniel Bernoulli and they are the canonical
solution of the differential equation
\[
u^{2}\frac{d^{2}J}{du^{2}}+u\frac{dJ}{du}+\left(u^{2}-v^{2}\right)J=0
\]
for any complex number $v$. In particular, equation (8) on page 177
of \cite{Wat} gives the Sonine representation
\begin{equation}
J_{\nu}\left(u\right)=\frac{\left(u/2\right)^{\nu}}{2\pi i}\int_{\left(a\right)}e^{s}s^{-\nu-1}e^{-u^{2}/\left(4s\right)}ds\label{eq:bes}
\end{equation}
where the notation $\int_{\left(a\right)}$ means $\int_{a-i\infty}^{a+i\infty}$. The method we will
use in this additive problem is based on a formula due to Laplace
\cite{Lap}, namely
\begin{equation}
\frac{1}{2\pi i}\int_{\left(a\right)}v^{-s}e^{v}dv=\frac{1}{\Gamma\left(s\right)}\label{eq:Lap}
\end{equation}
with $\textrm{Re}\left(s\right)>0$ and $a>0$ (see, e.g., formula
5.4 (1) on page 238 of \cite{Erd}). As in \cite{langzac2}, we combine
this approach with line integrals with the classical methods dealing
with infinite sum over primes and integers. Similarly as \cite{langzac2}
the problem naturally involves the modular relation for the complex
Jacobi $\theta_{3}$ function; the presence of the Bessel functions
in our statement strictly depends on such modularity relation.

\section{Preliminary definitions and Lemmas}
Let
\[
r_{Q}\left(n\right)=\sum_{m_{1}+m_{2}^{2}+m_{3}^{2}=n}\Lambda\left(m_{1}\right)
\]
and let $J_{v}\left(u\right)$ be the Bessel function of complex order
$v$ and real argument $u$.
Let $z=a+iy$, $a>0$, and 
\begin{align}
\theta_{3}\left(z\right)= & \sum_{m\in\mathbb{Z}}e^{-m^{2}z},\label{eq:rq}\\
\widetilde{S}\left(z\right)= & \sum_{m\geq1}\Lambda\left(m\right)e^{-mz},\label{eq:stild}\\
\omega_{2}\left(z\right)= & \sum_{m\geq1}e^{-m^{2}z},\label{eq:om2}
\end{align}
and we can see that 
\begin{equation}
\theta_{3}\left(z\right)=1+2\omega_{2}\left(z\right).\label{eq:omegaetheta}
\end{equation}
Furthermore we have the functional equation (see, for example, the
proposition VI.4.3 of Freitag-Busam \cite{frei} page 340)
\begin{equation}
\theta_{3}\left(z\right)=\left(\frac{\pi}{z}\right)^{1/2}\theta_{3}\left(\frac{\pi^{2}}{z}\right),\,\textrm{Re}\left(z\right)>0\label{eq:thetafunct}
\end{equation}
and so 

\begin{equation}
\omega_{2}^{2}\left(z\right)=\left(\frac{1}{2}\left(\frac{\pi}{z}\right)^{1/2}-\frac{1}{2}\right)^{2}+\frac{\pi}{z}\omega_{2}^{2}\left(\frac{\pi^{2}}{z}\right)+\left(\left(\frac{\pi}{z}\right)^{1/2}-1\right)\left(\left(\frac{\pi}{z}\right)^{1/2}\omega_{2}\left(\frac{\pi^{2}}{z}\right)\right).\label{eq:omfunct}
\end{equation}
A trivial but important estimate is 
\begin{equation}
\left|\omega_{2}\left(z\right)\right|\leq\omega_{2}\left(a\right)\leq\int_{0}^{\infty}e^{-at^{2}}dt=\frac{\sqrt{\pi}}{2\sqrt{a}}\ll a^{-1/2}.\label{eq:om2est}
\end{equation}
Let us introduce the following
\begin{lem}
Let $z=a+iy,\, a>0$ and $y\in\mathbb{R}$. Then
\begin{equation}
\widetilde{S}\left(z\right)=\frac{1}{z}-\sum_{\rho}z^{-\rho}\Gamma\left(\rho\right)+E\left(a,y\right)\label{eq:stilda}
\end{equation}
where $\rho=\beta+i\gamma$ runs over the non-trivial zeros of $\zeta\left(s\right)$
and
\[
E\left(a,y\right)\ll\left|z\right|^{1/2}\begin{cases}
1, & \left|y\right|\leq a\\
1+\log^{2}\left(\left|y\right|/a\right), & \left|y\right|>a.
\end{cases}
\]

\end{lem}
(For a proof see Lemma 1 of \cite{langzac1}. The bound for $E\left(a,y\right)$
has been corrected in \cite{Lang}). So in particular, taking $z=\frac{1}{N}+iy$
we have 
\[
\left|\sum_{\rho}z^{-\rho}\Gamma\left(\rho\right)\right|=\left|\frac{1}{z}-\widetilde{S}\left(z\right)+E\left(\frac{1}{N},y\right)\right|\ll N+\frac{1}{\left|z\right|}+\left|E\left(\frac{1}{N},y\right)\right|
\]
\begin{equation}
\ll\begin{cases}
N, & \left|y\right|\leq1/N\\
N+\left|z\right|^{1/2}\log^{2}\left(2N\left|y\right|\right), & \left|y\right|>1/N.
\end{cases}\label{eq:lemmaspec}
\end{equation}
Now we have to recall that the Prime Number Theorem (PNT) is equivalent,
via Lemma 2.1, to the statement
\[
\widetilde{S}\left(a\right)\sim a^{-1},\textrm{ when }a\rightarrow0^{+}
\]
(see Lemma 9 of \cite{HarLit}). For our purposes it is important
to introduce the Stirling approximation
\begin{equation}
\left|\Gamma\left(x+iy\right)\right|\sim\sqrt{2\pi}e^{-\pi\left|y\right|/2}\left|y\right|^{x-1/2}\label{eq:1}
\end{equation}
(see for example  \S 4.42 of \cite{titfun}) uniformly for $x\in\left[x_{1},x_{2}\right]$,
$x_{1}$ and $x_{2}$ fixed, and the identity 
\begin{equation}
\left|z^{-w}\right|=\left|z\right|^{-\textrm{Re}\left(w\right)}\exp\left(\textrm{Im}\left(w\right)\arctan\left(y/a\right)\right).\label{eq:modzcomplgen}
\end{equation}

We now quote Lemmas 2 and 3 from \cite{langzac1}:
\begin{lem}
Let $\beta+i\gamma$ run over the non-trivial zeros of the Riemann
zeta function and let $\alpha>1$ be a parameter. The series
\[
\sum_{\rho,\,\gamma>0}\gamma^{\beta-1/2}\int_{1}^{\infty}\exp\left(-\gamma\arctan\left(1/u\right)\right)\frac{dy}{u^{\alpha+\beta}}
\]
converges provided that $\alpha>3/2$. For $\alpha\leq3/2$ the series
does not converge. The result remains true if we insert in the integral
a factor $\log^{c}\left(u\right)$, for any fixed $c\geq0$.
\end{lem}
$\ $
\begin{lem}
Let $\beta+i\gamma$ run over the non-trivial zeros of the Riemann
zeta function, let $z=a+iy,$ $a\in\left(0,1\right)$, $y\in\mathbb{R}$
and $\alpha>1$. We have
\[
\sum_{\rho}\left|\gamma\right|^{\beta-1/2}\int_{\mathbb{Y}_{1}\cup\mathbb{Y}_{2}}\exp\left(\gamma\arctan\left(\frac{y}{a}\right)-\frac{\pi}{2}\left|\gamma\right|\right)\frac{dy}{\left|z\right|^{\alpha+\beta}}\ll_{\alpha}a^{-\alpha}
\]
where $\mathbb{Y}_{1}=\left\{ y\in\mathbb{R}:\,\gamma y\leq0\right\} $
and $\mathbb{Y}_{2}=\left\{ y\in\left[-a,a\right]:\, y\gamma>0\right\} $.
The result remains true if we insert in the integral a factor $\log^{c}\left(\left|y\right|/a\right)$,
for any fixed $c\geq0.$
\end{lem}
We now establish an important Lemma. We will use it to prove that
there is a limitation in our technique. Essentially the lower bound
of $k$ is linked to the number of squares in the problem. We have
\begin{lem}
Let $\beta+i\gamma$ run over the non-trivial zeros of the Riemann
zeta-function, let $N,\, d$ be positive integers, $\left\Vert .\right\Vert $
the euclidean norm in $\mathbb{R}^{d}$ and $k>0$ be a real number.
Then the series
\[
\sum_{\overline{l}\in\left(0,\infty\right)^{d}}\sum_{\gamma>0}\gamma^{-k-3/2}\int_{0}^{\gamma}e^{-N\left\Vert \overline{l}\right\Vert ^{2}v^{2}/\gamma^{2}}e^{-v}v^{k+\beta}dv,
\]
where
\[
\sum_{\overline{l}\in\left(0,\infty\right)^{d}}=\sum_{l_{1}\geq1}\sum_{l_{2}\geq1}\cdots\sum_{l_{d}\geq1},
\]
converges if $k>d-1/2$ and this result is optimal.\end{lem}
\begin{proof}
From (\ref{eq:omegaetheta}) we have that 
\[
\omega_{2}^{d}\left(z\right)=\frac{1}{2^{d}}\sum_{m=0}^{d}\dbinom{d}{m}\left(-1\right)^{d-m}\theta_{3}^{m}\left(z\right).
\]
Hence

\begin{align*}
I= & \sum_{\overline{l}\in\left(0,\infty\right)^{d}}\sum_{\gamma>0}\gamma^{-k-3/2}\int_{0}^{\gamma}e^{-N\left\Vert \overline{l}\right\Vert ^{2}v^{2}/\gamma^{2}}e^{-v}v^{k+\beta}dv\\
= & \sum_{\gamma>0}\gamma^{-k-3/2}\int_{0}^{\gamma}\omega_{2}^{d}\left(\frac{Nv^{2}}{\gamma^{2}}\right)e^{-v}v^{k+\beta}dv\\
= & \frac{1}{2^{d}}\sum_{m=0}^{d}\dbinom{d}{m}\left(-1\right)^{d-m}\sum_{\gamma>0}\gamma^{-k-3/2}\int_{0}^{\gamma}\theta_{3}^{m}\left(\frac{Nv^{2}}{\gamma^{2}}\right)e^{-v}v^{k+\beta}dv.
\end{align*}
Now, using the functional equation (\ref{eq:thetafunct}) we have
that

\begin{align*}
I= & \frac{1}{2^{d}}\sum_{m=0}^{d}\dbinom{d}{m}\left(-1\right)^{d-m}\frac{\pi^{m/2}}{N^{m/2}}\sum_{\gamma>0}\gamma^{m-k-3/2}\int_{0}^{\gamma}\theta_{3}^{m}\left(\frac{\pi^{2}\gamma^{2}}{Nv^{2}}\right)e^{-v}v^{k+\beta-m}dv\\
= & \frac{1}{2^{d}}\sum_{m=0}^{d}\dbinom{d}{m}\left(-1\right)^{d-m}\frac{\pi^{m/2}}{N^{m/2}}\sum_{\gamma>0}\gamma^{m-k-3/2}I_{\gamma,m},
\end{align*}
say. Now we claim that 
\[
\theta_{3}\left(\frac{\pi^{2}\gamma^{2}}{Nv^{2}}\right)\asymp1,
\]
where the notation $f\left(x\right)\asymp g\left(x\right)$ means
$g\left(x\right)\ll f\left(x\right)\ll g\left(x\right)$, since $\theta_{3}\left(x\right)$
is a continuous function in the interval $\left[\frac{\pi^{2}}{N},\infty\right)$
(i.e. the range of $1/v^{2}$) and 
\[
\lim_{x\rightarrow\infty}\theta_{3}\left(x\right)=1.
\]
So we have
\[
I_{\gamma,m}\asymp\sum_{\gamma>0}\gamma^{m-k-3/2}\int_{0}^{\gamma}e^{-v}v^{k+\beta-m}dv
\]
and now, assuming $k+\beta-m+1>0$, we get
\[
\int_{0}^{\gamma}e^{-v}v^{k+\beta-m}dv\asymp1.
\]
Hence
\[
I_{\gamma,m}\asymp_{k}\sum_{\gamma>0}\gamma^{m-k-3/2}
\]
and the last series converges if $k>m-1/2$. Since $m=0,\dots,d$
for a global convergence we must have $k>d-1/2$ and this result is
optimal.
\end{proof}
Let us introduce another lemma
\begin{lem}
Let $\rho=\beta+i\gamma$ run over the non-trivial zeros of the Riemann
zeta function, let $z=\frac{1}{N}+iy,$ $N>1$ natural number, $y\in\mathbb{R}$
and $\alpha>3/2$. We have
\[
\sum_{\rho}\left|\Gamma\left(\rho\right)\right|\int_{\left(1/N\right)}\left|e^{Nz}\right|\left|z^{-\rho}\right|\left|z\right|^{-\alpha}\left|dz\right|\ll_{\alpha}N^{\alpha}.
\]
\end{lem}
\begin{proof}
Put $a=\frac{1}{N}.$ Using the identity (\ref{eq:modzcomplgen})
and (\ref{eq:1}) we get that the left hand side in the statement
above is

\begin{equation}
\sum_{\rho}\left|\gamma\right|^{\beta-1/2}\int_{\mathbb{R}}\exp\left(\gamma\arctan\left(\frac{y}{a}\right)-\frac{\pi}{2}\left|\gamma\right|\right)\frac{dy}{\left|z\right|^{\alpha+\beta}}.\label{eq:lemma6}
\end{equation}
and so by Lemma 2.3 (\ref{eq:lemma6}) is $\ll_{\alpha}a^{-\alpha}$
in $\mathbb{Y}_{1}\cup\mathbb{Y}_{2}$. For the other part we can
see that
\[
\sum_{\rho}\gamma^{\beta-1/2}\int_{a}^{\infty}\exp\left(-\gamma\arctan\left(\frac{a}{y}\right)\right)\frac{dy}{\left|z\right|^{\alpha+\beta}}
\]
\[
=a^{-\alpha-\beta+1}\sum_{\rho}\gamma^{\beta-1/2}\int_{1}^{\infty}\exp\left(-\gamma\arctan\left(\frac{1}{u}\right)\right)\frac{dy}{u^{\alpha+\beta}}
\]
since 
\begin{equation}
\left|z\right|^{-1}\asymp\begin{cases}
a^{-1} & \left|y\right|\leq a,\\
\left|y\right|^{-1} & \left|y\right|\geq a,
\end{cases}\label{eq:modz-1}
\end{equation}
and so by Lemma 2.2 we have the convergence if $\alpha>3/2$.
\end{proof}

\section{Settings}

Using (\ref{eq:rq}), (\ref{eq:stild}) and (\ref{eq:om2}) it is
not hard to see that 
\[
\widetilde{S}\left(z\right)\omega_{2}^{2}\left(z\right)=\sum_{m_{1}\geq1}\sum_{m_{2}\geq1}\sum_{m_{3}\geq1}\Lambda\left(m_{1}\right)e^{-\left(m_{1}+m_{2}^{2}+m_{3}^{2}\right)z}=\sum_{n\geq1}r_{Q}\left(n\right)e^{-nz}.
\]
 Let $z=a+iy,\, a>0$ and let us consider
\[
\frac{1}{2\pi i}\int_{\left(a\right)}e^{Nz}z^{-k-1}\widetilde{S}\left(z\right)\omega_{2}^{2}\left(z\right)dz=\frac{1}{2\pi i}\int_{\left(a\right)}e^{Nz}z^{-k-1}\sum_{n\geq1}r_{Q}\left(n\right)e^{-nz}dz.
\]
Now we prove that we can exchange the integral with the series. From
(\ref{eq:om2est}) and the Prime Number Theorem in the form quoted
above we have
\[
\sum_{n\geq1}\left|r_{Q}\left(n\right)e^{-nz}\right|=\widetilde{S}\left(a\right)\omega_{2}^{2}\left(a\right)\ll a^{-2}
\]
hence

\begin{align*}
\int_{\left(a\right)}\left|e^{Nz}z^{-k-1}\right|\left|\widetilde{S}\left(z\right)\omega_{2}^{2}\left(z\right)\right|\left|dz\right|\ll & a^{-2}e^{Na}\left(\int_{-a}^{a}a^{-k-1}dy+2\int_{a}^{\infty}y^{-k-1}dy\right)\ll_{k}a^{-2-k}e^{Na}
\end{align*}
assuming $k>0$. So finally we have
\begin{equation}
\sum_{n\leq N}r_{Q}\left(n\right)\frac{\left(N-n\right)^{k}}{\Gamma\left(k+1\right)}=\frac{1}{2\pi i}\int_{\left(a\right)}e^{Nz}z^{-k-1}\widetilde{S}\left(z\right)\omega_{2}^{2}\left(z\right)dz.\label{eq:main}
\end{equation}
Now, using (\ref{eq:stilda}), we can write (\ref{eq:main}) as

\[
\sum_{n\leq N}r_{Q}\left(n\right)\frac{\left(N-n\right)^{k}}{\Gamma\left(k+1\right)}=\frac{1}{2\pi i}\int_{\left(a\right)}e^{Nz}z^{-k-1}\left(\frac{1}{z}-\sum_{\rho}z^{-\rho}\Gamma\left(\rho\right)\right)\omega_{2}^{2}\left(z\right)dz+
\]
\begin{equation}
+O\left(\int_{\left(a\right)}\left|e^{Nz}\right|\left|z\right|^{-k-1}\left|\omega_{2}^{2}\left(z\right)\right|\left|E\left(a,y\right)\right|\left|dz\right|\right)
\end{equation}
and the error term can be estimated, using Lemma 2.1, (\ref{eq:om2est})
and (\ref{eq:modz-1}) as 
\[
a^{-1}e^{Na}\left(\int_{-a}^{a}a^{-k-1}dy+\int_{a}^{\infty}y^{-k-1/2}\left(1+\log^{2}\left(y/a\right)\right)dy\right)\ll_{k}e^{Na}a^{-k-1}
\]
assuming $k>1/2.$ Hereafter we will consider $a=1/N$. We have 
\[
\sum_{n\leq N}r_{Q}\left(n\right)\frac{\left(N-n\right)^{k}}{\Gamma\left(k+1\right)}=\frac{1}{2\pi i}\int_{\left(1/N\right)}e^{Nz}z^{-k-1}\left(\frac{1}{z}-\sum_{\rho}z^{-\rho}\Gamma\left(\rho\right)\right)\omega_{2}^{2}\left(z\right)dz+O\left(N^{k+1}\right)
\]
and now, using the functional equation (\ref{eq:omfunct}), we get

\begin{align*}
\sum_{n\leq N}r_{Q}\left(n\right)\frac{\left(N-n\right)^{k}}{\Gamma\left(k+1\right)}= & \frac{1}{8\pi i}\int_{\left(1/N\right)}e^{Nz}z^{-k-1}\left(\frac{1}{z}-\sum_{\rho}z^{-\rho}\Gamma\left(\rho\right)\right)\left(\left(\frac{\pi}{z}\right)^{1/2}-1\right)^{2}dz\\
+ & \frac{1}{2\pi i}\int_{\left(1/N\right)}e^{Nz}z^{-k-1}\left(\frac{1}{z}-\sum_{\rho}z^{-\rho}\Gamma\left(\rho\right)\right)\frac{\pi}{z}\omega_{2}^{2}\left(\frac{\pi^{2}}{z}\right)dz\\
+ & \frac{1}{2\pi i}\int_{\left(1/N\right)}e^{Nz}z^{-k-1}\left(\frac{1}{z}-\sum_{\rho}z^{-\rho}\Gamma\left(\rho\right)\right)\left(\left(\frac{\pi}{z}\right)^{1/2}-1\right)\left(\left(\frac{\pi}{z}\right)^{1/2}\omega_{2}\left(\frac{\pi^{2}}{z}\right)\right)dz\\
+ & O\left(N^{k+1}\right)\\
= & I_{1}+I_{2}+I_{3}+O\left(N^{k+1}\right),
\end{align*}
say.

\section{Evaluation of $I_{1}$}

From $I_{1}$ we will find the main terms $M_{1}\left(N,k\right)$
and $M_{2}\left(N,k\right)$ of our asymptotic formulae. We have

\begin{align*}
I_{1}= & \frac{1}{8\pi i}\int_{\left(1/N\right)}e^{Nz}z^{-k-2}\left(\left(\frac{\pi}{z}\right)^{1/2}-1\right)^{2}dz\\
- & \frac{1}{8\pi i}\int_{\left(1/N\right)}e^{Nz}z^{-k-1}\sum_{\rho}z^{-\rho}\Gamma\left(\rho\right)\left(\left(\frac{\pi}{z}\right)^{1/2}-1\right)^{2}dz\\
= & I_{1,1}-I_{1,2},
\end{align*}
say. From $I_{1,1}$ we observe that 
\[
I_{1,1}=\frac{\pi}{8\pi i}\int_{\left(1/N\right)}e^{Nz}z^{-k-3}dz+\frac{1}{8\pi i}\int_{\left(1/N\right)}e^{Nz}z^{-k-2}dz-\frac{\pi^{1/2}}{4\pi i}\int_{\left(1/N\right)}e^{Nz}z^{-k-5/2}dz
\]
so, if we put $Nz=s,$ $ds=Ndz$ and use (\ref{eq:Lap}) we get immediately

\begin{align*}
I_{1,1}= & \frac{\pi}{4}\frac{N^{k+2}}{2\pi i}\int_{\left(1\right)}e^{s}s^{-k-3}ds+\frac{N^{k+1}}{4}\frac{1}{2\pi i}\int_{\left(1\right)}e^{s}s^{-k-2}ds-\frac{\pi}{2}\frac{N^{k+3/2}}{2\pi i}\int_{\left(1\right)}e^{s}s^{-k-5/2}ds\\
= & M_{1}\left(N,k\right).
\end{align*}
From $I_{1,2}$ we have

\begin{align*}
I_{1,2}= & \frac{\pi}{8\pi i}\int_{\left(1/N\right)}e^{Nz}z^{-k-2}\sum_{\rho}z^{-\rho}\Gamma\left(\rho\right)dz\\
+ & \frac{1}{8\pi i}\int_{\left(1/N\right)}e^{Nz}z^{-k-1}\sum_{\rho}z^{-\rho}\Gamma\left(\rho\right)dz\\
- & \frac{\pi^{1/2}}{4\pi i}\int_{\left(1/N\right)}e^{Nz}z^{-k-3/2}\sum_{\rho}z^{-\rho}\Gamma\left(\rho\right)dz\\
= & \mathcal{I}_{1}+\mathcal{I}_{2}-\mathcal{I}_{3},
\end{align*}
say. We observe that by Lemma 2.5 we have the absolute convergence of
these integrals if, respectively, we have $k>-1/2$, $k>1/2$ and
$k>0$. Hence for $k>1/2$ we have
\[
\mathcal{I}_{1}=\frac{\pi}{4}\sum_{\rho}\Gamma\left(\rho\right)\frac{1}{2\pi i}\int_{(1/N)}e^{Nz}z^{-k-2-\rho}dz=\frac{\pi}{4}\sum_{\rho}\frac{\Gamma\left(\rho\right)}{\Gamma\left(k+2+\rho\right)}N^{k+1+\rho}
\]
\[
\mathcal{I}_{2}=\frac{1}{4}\sum_{\rho}\Gamma\left(\rho\right)\frac{1}{2\pi i}\int_{(1/N)}e^{Nz}z^{-k-1-\rho}dz=\frac{1}{4}\sum_{\rho}\frac{\Gamma\left(\rho\right)}{\Gamma\left(k+1+\rho\right)}N^{k+\rho}
\]

\[
\mathcal{I}_{3}=\frac{\pi^{1/2}}{2}\sum_{\rho}\Gamma\left(\rho\right)\frac{1}{2\pi i}\int_{(1/N)}e^{Nz}z^{-k-3/2-\rho}dz=\frac{\pi^{1/2}}{2}\sum_{\rho}\frac{\Gamma\left(\rho\right)}{\Gamma\left(k+3/2+\rho\right)}N^{k+1/2+\rho}.
\]

\section{Evaluation of $I_{2}$}

We have

\begin{align*}
I_{2}= & \frac{\pi}{2\pi i}\int_{\left(1/N\right)}e^{Nz}z^{-k-3}\omega_{2}^{2}\left(\frac{\pi^{2}}{z}\right)dz\\
- & \frac{\pi}{2\pi i}\int_{\left(1/N\right)}e^{Nz}z^{-k-2}\sum_{\rho}z^{-\rho}\Gamma\left(\rho\right)\omega_{2}^{2}\left(\frac{\pi^{2}}{z}\right)dz\\
= & I_{2,1}-I_{2,2},
\end{align*}
say.

\subsection*{Evaluation of $\mathbf{I_{2,1}}$}

We have that 
\[
I_{2,1}:=\frac{\pi}{2\pi i}\int_{\left(1/N\right)}e^{Nz}z^{-k-3}\omega_{2}^{2}\left(\frac{\pi^{2}}{z}\right)dz=\frac{\pi}{2\pi i}\int_{\left(1/N\right)}e^{Nz}z^{-k-3}\left(\sum_{l_{1}\geq1}e^{-l_{1}^{2}\pi^{2}/z}\right)\left(\sum_{l_{2}\geq1}e^{-l_{2}^{2}\pi^{2}/z}\right)dz;
\]
so let us prove that we can exchange the integral with the series.
Let us consider
\[
A_{1}:=\sum_{l_{1}\geq1}\int_{\left(1/N\right)}\left|e^{Nz}\right|\left|z\right|^{-k-3}e^{-l_{1}^{2}\pi^{2}\textrm{Re}\left(1/z\right)}\left|\omega_{2}\left(\frac{\pi^{2}}{z}\right)\right|\left|dz\right|,
\]
say. From 
\begin{equation}
\textrm{Re}\left(1/z\right)=\frac{N}{1+N^{2}y^{2}}\gg\begin{cases}
N & \left|y\right|\leq1/N\\
1/\left(Ny^{2}\right) & \left|y\right|>1/N
\end{cases}\label{eq:re1/z}
\end{equation}
 we have 
\[
A_{1}\ll\sum_{l_{1}\geq1}\int_{0}^{1/N}\frac{e^{-l_{1}^{2}N}}{\left|z\right|^{k+3}}\omega_{2}\left(N\right)dy+N^{1/2}\sum_{l_{1}\geq1}\int_{1/N}^{\infty}\frac{ye^{-l_{1}^{2}/\left(Ny^{2}\right)}}{\left|z\right|^{k+3}}dy=U_{1}+U_{2}
\]
hence, recalling (\ref{eq:om2est}) and (\ref{eq:modz-1}), 
\[
U_{1}\ll N^{k+2}\omega_{2}^{2}\left(N\right)\ll N^{k+1}
\]
and from (\ref{eq:modz-1}) (with $a=1/N)$ we get 
\[
U_{2}\ll N^{1/2}\sum_{l_{1}\geq1}\int_{1/N}^{\infty}\frac{e^{-l_{1}^{2}/\left(Ny^{2}\right)}}{y^{k+2}}dy\ll N^{k/2+1}\sum_{l_{1}\geq1}\frac{1}{l_{1}^{k+1}}\int_{0}^{l_{1}^{2}N}u^{k/2-1/2}e^{-u}du
\]
\[
\leq\Gamma\left(\frac{k+1}{2}\right)N^{k/2+1}\sum_{l_{1}\geq1}\frac{1}{l_{1}^{k+1}}\ll_{k}N^{k/2+1}
\]
assuming $k>0.$ Now we have to study the convergence of 
\[
A_{2}:=\sum_{l_{1}\geq1}\sum_{l_{2}\geq1}\int_{\left(1/N\right)}\left|e^{Nz}\right|\left|z\right|^{-k-3}e^{-l_{1}^{2}\pi^{2}\textrm{Re}\left(1/z\right)}e^{-l_{2}^{2}\pi^{2}\textrm{Re}\left(1/z\right)}\left|dz\right|,
\]
say. Again from (\ref{eq:modz-1}) we have 
\[
A_{2}\ll\sum_{l_{1}\geq1}\sum_{l_{2}\geq1}\int_{0}^{1/N}\frac{e^{-\left(l_{1}^{2}+l_{2}^{2}\right)N}}{\left|z\right|^{k+3}}dy+\sum_{l_{1}\geq1}\sum_{l_{2}\geq1}\int_{1/N}^{\infty}\frac{e^{-\left(l_{1}^{2}+l_{2}^{2}\right)/\left(Ny^{2}\right)}}{\left|z\right|^{k+3}}dy
\]
\[
=V_{1}+V_{2},
\]
say. For $V_{1}$ we can repeat the same reasoning of $U_{1}$ thus
getting 
\[
V_{1}\ll N^{k+2}\omega_{2}^{2}\left(N\right)\ll N^{k+1}
\]
and for $V_{2}$, assuming $k>1$, we have
\[
V_{2}\ll\sum_{l_{1}\geq1}\sum_{l_{2}\geq1}\int_{1/N}^{\infty}\frac{e^{-\left(l_{1}^{2}+l_{2}^{2}\right)/\left(Ny^{2}\right)}}{y^{k+3}}dy\ll_{k}N^{k/2+1/2}.
\]
Then finally we have 
\[
I_{2,1}=\frac{\pi}{2\pi i}\sum_{l_{1}\geq1}\sum_{l_{2}\geq1}\int_{\left(1/N\right)}e^{Nz}z^{-k-3}e^{-\left(l_{1}^{2}+l_{2}^{2}\right)\pi^{2}/z}dz=N^{k+2}\pi\sum_{l_{1}\geq1}\sum_{l_{2}\geq1}\frac{1}{2\pi i}\int_{\left(1\right)}e^{s}s^{-k-3}e^{-\left(l_{1}^{2}+l_{2}^{2}\right)\pi^{2}N/s}ds
\]
 from which, recalling the definition of the Bessel functions (\ref{eq:bes})
we have, taking $u=2\pi\left(l_{1}^{2}+l_{2}^{2}\right)^{1/2}N^{1/2}$
and assuming $k>1$, that 
\[
I_{2,1}=\frac{N^{k/2+1}}{\pi^{k+1}}\sum_{l_{1}\geq1}\sum_{l_{2}\geq1}\frac{J_{k+2}\left(2\pi\left(l_{1}^{2}+l_{2}^{2}\right)^{1/2}N^{1/2}\right)}{\left(l_{1}^{2}+l_{2}^{2}\right)^{k/2+1}}.
\]

\subsection*{Evaluation of $\mathbf{I_{2,2}}$}

We have to calculate 
\[
I_{2,2}:=\frac{\pi}{2\pi i}\int_{\left(1/N\right)}e^{Nz}z^{-k-2}\sum_{\rho}z^{-\rho}\Gamma\left(\rho\right)\left(\sum_{l_{1}\geq1}e^{-l_{1}^{2}\pi^{2}/z}\right)\left(\sum_{l_{2}\geq1}e^{-l_{2}^{2}\pi^{2}/z}\right)dz
\]
and again we have to prove that is possible to exchange the integral
with the series. So let us consider
\[
A_{3}:=\sum_{l_{1}\geq1}\int_{\left(1/N\right)}\left|e^{Nz}\right|\left|z^{-k-2}\right|\left|\sum_{\rho}z^{-\rho}\Gamma\left(\rho\right)\right|e^{-l_{1}^{2}\pi^{2}\textrm{Re}\left(1/z\right)}\left|\omega_{2}\left(\frac{\pi^{2}}{z}\right)\right|\left|dz\right|,
\]
say. Now using (\ref{eq:lemmaspec}) and (\ref{eq:om2est}) we have
\[
A_{3}\ll N^{1/2}\sum_{l_{1}\geq1}\int_{0}^{1/N}\frac{e^{-l_{1}^{2}N}}{\left|z\right|^{k+2}}dy+N^{3/2}\sum_{l_{1}\geq1}\int_{1/N}^{\infty}\frac{ye^{-l_{1}^{2}/\left(Ny^{2}\right)}}{\left|z\right|^{k+2}}dy+N^{1/2}\sum_{l_{1}\geq1}\int_{1/N}^{\infty}y\log^{2}\left(2Ny\right)\frac{e^{-l_{1}^{2}/\left(Ny^{2}\right)}}{\left|z\right|^{k+3/2}}dy
\]
\[
=W_{1}+W_{2}+W_{3},
\]
say. For $W_{1}$ and $W_{2}$ we can easily see that
\[
W_{1}\ll N^{k+3/2}\omega_{2}\left(N\right)\ll N^{k+1}
\]
and, taking $u=l_{1}^{2}/\left(Ny^{2}\right),$ we obtain 
\[
W_{2}\ll N^{3/2}\sum_{l_{1}\geq1}\int_{1/N}^{\infty}\frac{e^{-l_{1}^{2}/\left(Ny^{2}\right)}}{y^{k+1}}dy
\]
\[
\ll N^{k/2+3/2}\sum_{l_{1}\geq1}\frac{1}{l_{1}^{k}}\int_{0}^{l_{1}^{2}N}e^{-u}u^{k/2-1}du\ll_{k}N^{k/2+3/2}
\]
assuming $k>1.$ We have now to check $W_{3}$. Taking again $u=l_{1}^{2}/\left(Ny^{2}\right)$
we have, assuming $k>3/2$, that
\begin{align*}
W_{3}\ll & N^{k/2-1/4}\sum_{l_{1}\geq1}\frac{1}{l_{1}^{k-1/2}}\int_{0}^{l_{1}^{2}N}\log^{2}\left(\frac{4Nl_{1}^{2}}{u}\right)e^{-u}u^{k/2-5/4}du\\
\ll & N^{k/2-1/4}\sum_{l_{1}\geq1}\frac{1}{l_{1}^{k-1/2}}\ll_{k}N^{k/2}.
\end{align*}
Let us consider 
\[
A_{4}:=\sum_{l_{1}\geq1}\sum_{l_{2}\geq2}\int_{\left(1/N\right)}\left|e^{Nz}\right|\left|z^{-k-2}\right|\left|\sum_{\rho}z^{-\rho}\Gamma\left(\rho\right)\right|e^{-l_{1}^{2}\pi^{2}\textrm{Re}\left(1/z\right)}e^{-l_{2}^{2}\pi^{2}\textrm{Re}\left(1/z\right)}\left|dz\right|,
\]
say. By (\ref{eq:lemmaspec}) we get 
\begin{align*}
A_{4}\ll & N\sum_{l_{1}\geq1}\sum_{l_{2}\geq2}\int_{0}^{1/N}\frac{e^{-\left(l_{1}^{2}+l_{2}^{2}\right)N}}{\left|z\right|^{k+2}}dy+\sum_{l_{1}\geq1}\sum_{l_{2}\geq2}\int_{1/N}^{\infty}\frac{e^{-\left(l_{1}^{2}+l_{2}^{2}\right)/\left(Ny^{2}\right)}}{\left|z\right|^{k+2}}dy\\
+ & \sum_{l_{1}\geq1}\sum_{l_{2}\geq1}\int_{1/N}^{\infty}\log^{2}\left(2Ny\right)\frac{e^{-\left(l_{1}^{2}+l_{2}^{2}\right)/\left(Ny^{2}\right)}}{\left|z\right|^{k+3/2}}dy\\
= & R_{1}+R_{2}+R_{3},
\end{align*}
say. So we have immediately 
\[
R_{1}\ll N^{k+2}\omega^{2}\left(N\right)\ll N^{k+1}
\]
and, if we take $u=\left(l_{1}^{2}+l_{2}^{2}\right)/\left(Ny^{2}\right),$
we obtain
\begin{align*}
R_{2}\ll & \sum_{l_{1}\geq1}\sum_{l_{2}\geq1}\int_{1/N}^{\infty}\frac{e^{-\left(l_{1}^{2}+l_{2}^{2}\right)/\left(Ny^{2}\right)}}{y^{k+2}}dy\ll_{k}N^{(k+1)/2}
\end{align*}
for $k>1$. So it remains to evaluate $R_{3}$. Again we take $u=\left(l_{1}^{2}+l_{2}^{2}\right)/\left(Ny^{2}\right)$
and we have

\[
R_{3}\ll N^{k/2+1/4}\sum_{l_{1}\geq1}\sum_{l_{2}\geq1}\frac{\log^{2}\left(4N\left(l_{1}^{2}+l_{2}^{2}\right)\right)}{\left(l_{1}^{2}+l_{2}^{2}\right)^{k/2+1/4}}\int_{0}^{\left(l_{1}^{2}+l_{2}^{2}\right)^{1/2}N}e^{-u}u^{k/2-3/4}du
\]
\[
-N^{k/2+1/4}\sum_{l_{1}\geq1}\sum_{l_{2}\geq1}\frac{1}{\left(l_{1}^{2}+l_{2}^{2}\right)^{k/2+1/4}}\int_{0}^{\left(l_{1}^{2}+l_{2}^{2}\right)^{1/2}N}\log^{2}\left(u\right)e^{-u}u^{k/2-3/4}du
\]
and the convergence follows if $k>3/2.$ Note that the estimation
of $R_{3}$ is optimal. For proving it, take $c=\left(l_{1}^{2}+l_{2}^{2}\right)/N$,
assume $k\leq3/2$ and $y>1$. We have
\[
S:=\sum_{l_{1}\geq1}\sum_{l_{2}\geq1}\int_{1/N}^{\infty}\log^{2}\left(2Ny\right)\frac{e^{-c/y^{2}}}{y^{k+3/2}}dy\geq\sum_{l_{1}\geq1}\sum_{l_{2}\geq1}\int_{1}^{\infty}\log^{2}\left(2Ny\right)\frac{e^{-c/y^{2}}}{y^{k+3/2}}dy.
\]
Now, since $y\geq1$ we have $\log^{2}\left(2Ny\right)\geq\log^{2}\left(2N\right)$
and since $k\leq3/2$, we have 
\[
S\geq\log\left(2N\right)\sum_{l_{1}\geq1}\sum_{l_{2}\geq1}\int_{1}^{\infty}\frac{e^{-c/y^{2}}}{y^{k+3/2}}dy\geq\log\left(2N\right)\sum_{l_{1}\geq1}\sum_{l_{2}\geq1}\int_{1}^{\infty}\frac{e^{-c/y^{2}}}{y^{3}}dy
\]
\[
=\log\left(2N\right)\sum_{l_{1}\geq1}\sum_{l_{2}\geq1}\frac{1}{2c}\left(1-e^{-c}\right)\geq\frac{N\log\left(2N\right)\left(1-e^{-2/N}\right)}{2}\sum_{l_{1}\geq1}\sum_{l_{2}\geq1}\frac{1}{l_{1}^{2}+l_{2}^{2}}.
\]
The last double series diverges since 
\[
\sum_{l_{1}\geq1}\sum_{l_{2}\geq1}\frac{1}{l_{1}^{2}+l_{2}^{2}}\geq\sum_{l_{1}\geq1}\sum_{1\leq l_{2}\leq l_{1}}\frac{1}{l_{1}^{2}+l_{2}^{2}}\geq\frac{1}{2}\sum_{l_{1}\geq1}\frac{1}{l_{1}}.
\]
 Now we have to estimate 
\[
A_{5}:=\sum_{l_{1}\geq1}\sum_{l_{2}\geq1}\sum_{\rho}\left|\Gamma\left(\rho\right)\right|\int_{\left(1/N\right)}\left|e^{Nz}\right|\left|z^{-k-2}\right|\left|z^{-\rho}\right|e^{-l_{1}^{2}\pi^{2}\textrm{Re}\left(1/z\right)}e^{-l_{2}^{2}\pi^{2}\textrm{Re}\left(1/z\right)}\left|dz\right|,
\]
say. Using (\ref{eq:1}) and (\ref{eq:modzcomplgen}) we have 
\[
A_{5}\ll\sum_{l_{1}\geq1}\sum_{l_{2}\geq1}\sum_{\rho,\gamma>0}e^{-\pi\gamma/2}\gamma^{\beta-1/2}\int_{1/N}^{\infty}\left|z\right|^{-k-2}\left|z\right|^{-\beta}\exp\left(\gamma\arctan\left(Ny\right)\right)e^{-l_{1}^{2}\pi^{2}\textrm{Re}\left(1/z\right)}e^{-l_{2}^{2}\pi^{2}\textrm{Re}\left(1/z\right)}\left|dz\right|.
\]
 Let $Q_{k}=\sup_{\beta}\left\{ \Gamma\left(\frac{k}{2}+\frac{\beta}{2}+\frac{1}{2}\right)\right\} $
and assume $y<0$. Using the trivial bound $\gamma\arctan\left(Ny\right)-\gamma\frac{\pi}{2}\leq-\gamma\frac{\pi}{2},$
we have
\begin{align}
A_{5}\ll & N^{k+1}\sum_{l_{1}\geq1}e^{-l_{1}^{2}N}\sum_{l_{2}\geq1}e^{-l_{2}^{2}N}\sum_{\rho,\gamma>0}N^{\beta}e^{-\pi\gamma/2}\gamma^{\beta-1/2}\nonumber \\
+ & N^{\left(k+1\right)/2}Q_{k}\sum_{l_{1}\geq1}\sum_{l_{2}\geq1}\frac{1}{\left(l_{1}^{2}+l_{2}^{2}\right)^{\left(k+1\right)/2}}\sum_{\rho,\gamma>0}N^{\beta}\frac{e^{-\pi\gamma/2}\gamma^{\beta-1/2}}{\left(l_{1}^{2}+l_{2}^{2}\right)^{\beta}}\ll_{k}N^{k}.\label{eq:imaginarypartzeta-1}
\end{align}
If $y>0$ we have
\begin{align*}
A_{5}\ll & \sum_{l_{1}\geq1}\sum_{l_{2}\geq1}\sum_{\rho:\gamma>0}e^{-\pi\gamma/2}\gamma^{\beta-1/2}\int_{0}^{1/N}N^{k+2+\beta}e^{-\left(l_{1}^{2}+l_{2}^{2}\right)N}dy\\
+ & \sum_{l_{1}\geq1}\sum_{l_{2}\geq1}\sum_{\rho:\gamma>0}\gamma^{\beta-1/2}\int_{1/N}^{\infty}\exp\left(\gamma\left(\arctan\left(Ny\right)-\frac{\pi}{2}\right)\right)\frac{e^{-\left(l_{1}^{2}+l_{2}^{2}\right)/\left(Ny^{2}\right)}}{y^{k+2+\beta}}dy
\end{align*}
 and by a well-known trigonometric identity follows that
\begin{align*}
A_{5}\ll & N^{k+1}+\sum_{l_{1}\geq1}\sum_{l_{2}\geq1}\sum_{\rho:\gamma>0}\gamma^{\beta-1/2}\int_{1/N}^{\infty}\exp\left(-\gamma\arctan\left(\frac{1}{Ny}\right)\right)\frac{e^{-\left(l_{1}^{2}+l_{2}^{2}\right)/\left(Ny^{2}\right)}}{y^{k+2+\beta}}dy\\
\ll & N^{k+1}+\sum_{l_{1}\geq1}\sum_{l_{2}\geq1}\sum_{\rho:\gamma>0}\gamma^{\beta-1/2}\int_{1/N}^{\infty}\exp\left(-\frac{\gamma}{Ny}-\frac{l_{1}^{2}+l_{2}^{2}}{Ny^{2}}\right)y^{-k-2-\beta}dy
\end{align*}
 and if we put $\frac{\gamma}{Ny}=v$ we get 

\begin{align}
A_{5}\ll & N^{k+1}+\sum_{l_{1}\geq1}\sum_{l_{2}\geq1}\sum_{\rho:\gamma>0}\gamma^{\beta-1/2}\int_{0}^{\gamma}e^{-v}e^{-\left(Nv^{2}\left(l_{1}^{2}+l_{2}^{2}\right)/\gamma^{2}\right)}\left(\frac{\gamma}{Nv}\right)^{-k-2-\beta}\frac{\gamma}{Nv^{2}}dv\nonumber \\
\ll & N^{k+1}+\sum_{l_{1}\geq1}\sum_{l_{2}\geq1}\sum_{\rho:\gamma>0}\gamma^{-k-3/2}\int_{0}^{\infty}e^{-v}e^{-\left(Nv^{2}\left(l_{1}^{2}+l_{2}^{2}\right)/\gamma^{2}\right)}v^{k+\beta}dv.\label{eq:k>3/2}
\end{align}
Now we can observe that we are in the situation of Lemma 2.4 with $d=2$
and so we can conclude immediately that we have the convergence for
$k>3/2$ and this result is optimal.

We studied the convergence, so we finally have, using again the identity
(\ref{eq:bes}), that

\[
I_{2,2}=\pi^{-k}N^{k/2+1/2}\sum_{\rho}\frac{\Gamma\left(\rho\right)}{\pi^{\rho}}N^{\rho/2}\sum_{l_{1}\geq1}\sum_{l_{2}\geq1}\frac{J_{k+1+\rho}\left(2\pi\left(l_{1}^{2}+l_{2}^{2}\right)^{1/2}N^{1/2}\right)}{\left(l_{1}^{2}+l_{2}^{2}\right)^{\left(k+1+\rho\right)/2}}.
\]

\section{Evaluation of $I_{3}$}

We have 
\begin{align*}
I_{3}= & \frac{1}{2\pi i}\int_{\left(1/N\right)}e^{Nz}z^{-k-1}\left(\frac{\pi^{1/2}}{z^{3/2}}-\left(\frac{\pi}{z}\right)^{1/2}\sum_{\rho}z^{-\rho}\Gamma\left(\rho\right)-\frac{1}{z}+\sum_{\rho}z^{-\rho}\Gamma\left(\rho\right)\right)\left(\left(\frac{\pi}{z}\right)^{1/2}\omega_{2}\left(\frac{\pi^{2}}{z}\right)\right)dz\\
= & \frac{1}{2i}\int_{\left(1/N\right)}e^{Nz}z^{-k-3}\omega_{2}\left(\frac{\pi^{2}}{z}\right)dz-\frac{1}{2i}\int_{\left(1/N\right)}e^{Nz}z^{-k-2}\sum_{\rho}z^{-\rho}\Gamma\left(\rho\right)\omega_{2}\left(\frac{\pi^{2}}{z}\right)dz\\
- & \frac{1}{2\pi^{1/2}i}\int_{\left(1/N\right)}e^{Nz}z^{-k-5/2}\omega_{2}\left(\frac{\pi^{2}}{z}\right)+\frac{1}{2\pi^{1/2}i}\int_{\left(1/N\right)}e^{Nz}z^{-k-3/2}\sum_{\rho}z^{-\rho}\Gamma\left(\rho\right)\omega_{2}\left(\frac{\pi^{2}}{z}\right)dz\\
= & I_{3,1}-I_{3,2}-I_{3,3}+I_{3,4}.
\end{align*}

\subsection*{Evaluation of $\mathbf{I_{3,1}}$}

We have
\[
I_{3,1}:=\frac{1}{2i}\int_{\left(1/N\right)}e^{Nz}z^{-k-3}\omega_{2}\left(\frac{\pi^{2}}{z}\right)dz=\frac{1}{2i}\int_{\left(1/N\right)}e^{Nz}z^{-k-3}\sum_{m\geq1}e^{-m^{2}\pi^{2}/z}dz
\]
hence we have to establish the convergence of 
\[
A_{6}:=\sum_{m\geq1}\int_{\left(1/N\right)}\left|e^{Nz}\right|\left|z\right|^{-k-3}e^{-m^{2}\textrm{Re}\left(1/z\right)}\left|dz\right|,
\]
say. Using (\ref{eq:om2est}), (\ref{eq:modz-1}) and (\ref{eq:re1/z})
we have 

\begin{align}
A_{6}\ll & N^{k+3/2}+\sum_{m\geq1}\int_{0}^{\infty}y^{-k-3}e^{-m^{2}/\left(Ny^{2}\right)}dy\ll_{k}N^{k+3/2}\label{eq:ultima}
\end{align}
for $k>-1.$ So we obtain, recalling (\ref{eq:bes}), that
\[
J_{3,1}=\frac{N^{k/2+1}}{\pi^{k+1}}\sum_{m\geq1}\frac{J_{k+2}\left(2m\pi N^{1/2}\right)}{m^{k+2}}.
\]

\subsection*{Evaluation of $\mathbf{I_{3,3}}$}

We have 
\[
I_{3,3}:=\frac{1}{2\pi^{1/2}i}\int_{\left(1/N\right)}e^{Nz}z^{-k-5/2}\sum_{m\geq1}e^{-m^{2}\pi^{2}/z}dz
\]
so we have to establish the convergence of
\[
\sum_{m\geq1}\int_{\left(1/N\right)}\left|e^{Nz}\right|\left|z\right|^{-k-5/2}e^{-m^{2}\textrm{Re}\left(1/z\right)}\left|dz\right|.
\]
Arguing as for $I_{3,1},$ we have the convergence for $k>-1/2$.
Summing up, we obtain

\begin{align*}
I_{3,3}= & \frac{N^{k/2+3/4}}{\pi^{k+1}}\sum_{m\geq1}\frac{J_{k+3/2}\left(2m\pi N^{1/2}\right)}{m^{k+3/2}}.
\end{align*}

\subsection*{Evaluation of $\mathbf{I_{3,2}}$}

We have to establish the convergence of
\[
A_{7}:=\sum_{m\geq1}\int_{\left(1/N\right)}\left|e^{Nz}\right|\left|z^{-k-2}\right|\left|\sum_{\rho}z^{-\rho}\Gamma\left(\rho\right)\right|\left|e^{-m^{2}\pi^{2}/z}\right|\left|dz\right|,
\]
say. Using (\ref{eq:om2est}), (\ref{eq:modz-1}), (\ref{eq:re1/z})
and (\ref{eq:lemmaspec}) we get 

\begin{align*}
A_{7}\ll & N^{k+1/2}+N\sum_{m\geq1}\int_{1/N}^{\infty}y^{-k-2}e^{-m^{2}/\left(Ny^{2}\right)}dy\\
+ & \log^{2}\left(2N\right)\sum_{m\geq1}\int_{1/N}^{\infty}y^{-k-3/2}e^{-m^{2}/\left(Ny^{2}\right)}dy\\
+ & \sum_{m\geq1}\int_{1/N}^{\infty}\log^{2}\left(y\right)y^{-k-3/2}e^{-m^{2}/\left(Ny^{2}\right)}dy.
\end{align*}
Now if we put $m^{2}/\left(Ny^{2}\right)=u$ we have

\begin{align*}
N\sum_{m\geq1}\int_{1/N}^{\infty}y^{-k-2}e^{-m^{2}/\left(Ny^{2}\right)}dy\ll & N^{k/2+3/2}\Gamma\left(\frac{k+1}{2}\right)\sum_{m\geq1}m^{-k-1}
\end{align*}
which converges if $k>0$. With the same substitution we get

\begin{align*}
\log^{2}\left(2N\right)\sum_{m\geq1}\int_{1/N}^{\infty}y^{-k-3/2}e^{-m^{2}/\left(Ny^{2}\right)}dy\ll & \log^{2}\left(2N\right)N^{k/2+1/4}\Gamma\left(\frac{k}{2}+\frac{1}{4}\right)\sum_{m\geq1}m^{-k-1/2}
\end{align*}
which converges for $k>1/2$. For the estimation of the last integral
in the bound of $A_{7}$ we observe that if we take $\epsilon>0$
we have 
\[
\sum_{m\geq1}\int_{1/N}^{\infty}\log^{2}\left(y\right)y^{-k-3/2}e^{-m^{2}/\left(Ny^{2}\right)}dy\ll\sum_{m\geq1}\int_{1/N}^{\infty}y^{-k-3/2+\epsilon}e^{-m^{2}/\left(Ny^{2}\right)}dy
\]
 and so, arguing analogously as we did for (\ref{eq:ultima}), we
get
\[
\ll N^{k/2+1/4-\epsilon/2}\Gamma\left(\frac{k}{2}+\frac{1}{4}-\frac{\epsilon}{2}\right)\sum_{m\geq1}m^{-k-1/2+\epsilon}
\]
and for the arbitrariness of $\epsilon$ we have the convergence for
$k>1/2$. We have now to study
\[
A_{8}:=\sum_{m\geq1}\sum_{\rho}\left|\Gamma\left(\rho\right)\right|\int_{\left(1/N\right)}\left|e^{Nz}\right|\left|z^{-k-2}\right|\left|z^{-\rho}\right|\left|e^{-m^{2}\pi^{2}/z}\right|\left|dz\right|,
\]
say. By symmetry we may assume that $\gamma>0$. If $y\leq0$ we have
$\gamma\arctan\left(y/a\right)-\frac{\pi}{2}\gamma\leq-\frac{\pi}{2}\gamma$
and so using (\ref{eq:1}) and (\ref{eq:modzcomplgen}) we get
\[
A_{8}\ll\sum_{m\geq1}\sum_{\gamma>0}\gamma^{\beta-1/2}\exp\left(-\frac{\pi}{2}\gamma\right)\left(\int_{-1/N}^{0}N^{k+2+\beta}e^{-m^{2}N}dy+\int_{-\infty}^{-1/N}\frac{e^{-m^{2}/\left(Ny^{2}\right)}}{\left|y\right|^{k+2+\beta}}dy\right)
\]
\[
\ll_{k}N^{k+3/2}+N^{k/2+1/2}Q_{k}\sum_{m\geq1}\frac{1}{m^{k+1}}\sum_{\gamma>0}N^{\beta/2}\frac{\gamma^{\beta-1/2}}{m^{\beta}}\exp\left(-\frac{\pi}{2}\gamma\right)\ll_{k}N^{k+3/2}
\]
provided that $k>0$ and $Q_{k}=\sup_{\beta}\left\{ \Gamma\left(\frac{k}{2}+\frac{1}{2}+\frac{\beta}{2}\right)\right\} .$
Let $y>0$. We have

\begin{align*}
A_{8}\ll & \sum_{m\geq1}\sum_{\gamma>0}\gamma^{\beta-1/2}\exp\left(-\frac{\pi}{4}\gamma\right)\int_{0}^{1/N}N^{k+2+\beta}e^{-m^{2}N}dy\\
+ & \sum_{m\geq1}\sum_{\gamma>0}\gamma^{\beta-1/2}\int_{1/N}^{\infty}\exp\left(\gamma\arctan\left(Ny\right)-\frac{\pi}{2}\gamma\right)\frac{e^{-m^{2}/\left(Ny^{2}\right)}}{y^{k+2+\beta}}dy\\
= & L_{1}+L_{2},
\end{align*}
say. From (\ref{eq:om2est}) and (\ref{eq:modz-1}) we have 
\[
L_{1}\ll N^{k+1}\sum_{m\geq1}e^{-m^{2}N}\sum_{\gamma>0}N^{\beta}\gamma^{\beta-1/2}\exp\left(-\frac{\pi}{4}\gamma\right)\ll_{k}N^{k+3/2}
\]
and again by a well-known trigonometric identity and taking $v=m/\left(N^{1/2}y\right)$
we have
\begin{align*}
L_{2}\ll & \sum_{m\geq1}\sum_{\gamma>0}\gamma^{\beta-1/2}\int_{1/N}^{\infty}\exp\left(-\frac{\gamma}{Ny}-\frac{m^{2}}{Ny^{2}}\right)\frac{dy}{y^{k+2+\beta}}\\
= & N^{\left(k+1\right)/2}\sum_{m\geq1}\frac{1}{m^{k+1}}\sum_{\gamma>0}\frac{N^{\beta/2}}{m^{\beta}}\gamma^{\beta-1/2}\int_{0}^{m\sqrt{N}}\exp\left(-\frac{\gamma v}{N^{1/2}m}-v^{2}\right)v^{k+\beta}dv.
\end{align*}
 Using $e^{-v^{2}}v^{k}=O_{k}\left(1\right)$ if $k>0$, we have,
taking $s=\gamma v/\left(N^{1/2}m\right)$, that

\begin{align*}
\ll & N^{k/2+1}\sum_{m\geq1}\frac{1}{m^{k}}\sum_{\gamma>0}N^{\beta}\gamma^{-3/2}\int_{0}^{\infty}\exp\left(-s\right)s^{\beta}ds\ll_{k}N^{k/2+2}
\end{align*}
for $k>1.$ Now we can exchange the series with the integral and so
we have

\begin{align*}
I_{3,2}= & \pi^{-k}N^{\left(k+1\right)/2}\sum_{\rho}\pi^{-\rho}N^{\rho/2}\Gamma\left(\rho\right)\sum_{m\geq1}\frac{J_{k+1+\rho}\left(2m\pi\sqrt{N}\right)}{m^{k+1+\rho}}.
\end{align*}

\subsection*{Evaluation of $\mathbf{I_{3,4}}$}

We have to establish the convergence of 
\[
I_{3,4}:=\frac{1}{2\pi^{1/2}i}\int_{\left(1/N\right)}e^{Nz}z^{-k-3/2}\sum_{\rho}z^{-\rho}\Gamma\left(\rho\right)\omega_{2}\left(\frac{\pi^{2}}{z}\right)dz.
\]
Arguing analogously as we did for estimating $I_{3,2}$ we obtain
the condition $k>1$. We can exchange the series with the integral
and obtain

\begin{align*}
I_{3,4}= & \pi^{-k}N^{k/2+1/4}\sum_{\rho}\pi^{-\rho}N^{\rho}\Gamma\left(\rho\right)\sum_{m\geq1}\frac{J_{k+1/2+\rho}\left(2m\pi\sqrt{N}\right)}{m^{k+1/2+\rho}}.
\end{align*}
Defining \begin{align}
M_{1}\left(N,k\right)= & \frac{\pi N^{k+2}}{4\Gamma\left(k+3\right)}+\frac{N^{k+1}}{4\Gamma\left(k+2\right)}-\frac{\pi^{1/2}N^{k+3/2}}{2\Gamma\left(k+5/2\right)},\label{eq:M1}\\
M_{2}\left(N,k\right)= & -\frac{\pi}{4}\sum_{\rho}\frac{\Gamma\left(\rho\right)}{\Gamma\left(k+2+\rho\right)}N^{k+1+\rho}-\frac{1}{4}\sum_{\rho}\frac{\Gamma\left(\rho\right)}{\Gamma\left(k+1+\rho\right)}N^{k+\rho}\nonumber \\
+ & \frac{\pi^{1/2}}{2}\sum_{\rho}\frac{\Gamma\left(\rho\right)}{\Gamma\left(k+3/2+\rho\right)}N^{k+1/2+\rho},\label{eq:M2}\\
M_{3}\left(N,k\right)= & \frac{N^{k/2+1}}{\pi^{k+1}}\sum_{l_{1}\geq1}\sum_{l_{2}\geq1}\frac{J_{k+2}\left(2\pi\left(l_{1}^{2}+l_{2}^{2}\right)^{1/2}N^{1/2}\right)}{\left(l_{1}^{2}+l_{2}^{2}\right)^{k/2+1}}\nonumber \\
- & \pi^{-k}N^{k/2+1/2}\sum_{\rho}\frac{\Gamma\left(\rho\right)}{\pi^{\rho}}N^{\rho/2}\sum_{l_{1}\geq1}\sum_{l_{2}\geq1}\frac{J_{k+1+\rho}\left(2\pi\left(l_{1}^{2}+l_{2}^{2}\right)^{1/2}N^{1/2}\right)}{\left(l_{1}^{2}+l_{2}^{2}\right)^{\left(k+1+\rho\right)/2}},\label{eq:M3}\\
M_{4}\left(N,k\right)= & \frac{N^{k/2+1}}{\pi^{k+1}}\sum_{m\geq1}\frac{J_{k+2}\left(2m\pi N^{1/2}\right)}{m^{k+2}}-\frac{N^{k/2+3/4}}{\pi^{k+1}}\sum_{m\geq1}\frac{J_{k+3/2}\left(2m\pi N^{1/2}\right)}{m^{k+3/2}}\nonumber \\
- & \pi^{-k}N^{\left(k+1\right)/2}\sum_{\rho}\pi^{-\rho}N^{\rho/2}\Gamma\left(\rho\right)\sum_{m\geq1}\frac{J_{k+1+\rho}\left(2m\pi\sqrt{N}\right)}{m^{k+1+\rho}}\nonumber \\
+ & \pi^{-k}N^{k/2+1/4}\sum_{\rho}\pi^{-\rho}N^{\rho/2}\Gamma\left(\rho\right)\sum_{m\geq1}\frac{J_{k+1/2+\rho}\left(2m\pi\sqrt{N}\right)}{m^{k+1/2+\rho}},\label{eq:M4}
\end{align}
we have proved the following
\begin{mainthm}
Let $N$ be a sufficient large integer. We have
\[
\sum_{n\leq N}r_{Q}\left(n\right)\frac{\left(N-n\right)^{k}}{\Gamma\left(k+1\right)}=M_{1}\left(N,k\right)+M_{2}\left(N,k\right)+M_{3}\left(N,k\right)+M_{4}\left(N,k\right)+O\left(N^{k+1}\right)
\]
for $k>3/2$, where $\rho$ runs over the non-trivial zeros of the
Riemann zeta function $\zeta\left(s\right)$ and $J_{v}\left(u\right)$
is the Bessel function of complex order $v$ and real argument $u$.
Furthermore the bound $k>3/2$ is optimal using this technique.
\end{mainthm}

\subsection*{Acknowledgements}
I thank my advisor A. Zaccagnini and A. Languasco for their contributions and the conversations on this topic and Lior Silberman of Mathoverflow.net for his precious ideas for Lemma 2.4. This work is part of the Author's Ph.D. thesis.

\end{document}